\documentclass[11pt, a4paper]{amsart}

\usepackage{amsfonts,amsmath,amssymb,amscd}

\newtheorem{theorem}{Theorem}[section]
\newtheorem{lemma}[theorem]{Lemma}
\newtheorem{prop}[theorem]{Proposition}
\newtheorem{corollary}[theorem]{Corollary}

\theoremstyle{definition}

\newtheorem{rem}[theorem]{Remark}

\newtheorem{exam}[theorem]{Example}

\title{Solvable, reductive and quasireductive supergroups}

\author{A.~N.~Grishkov and A.~N.~Zubkov}
\address{Alexander N.~Zubkov: 
Department of Mathematics, 
Omsk State Pedagogical University, 
Omsk--644099, Russia}
\email{a.zubkov@yahoo.com}
\address{Alexandr N.~Grishkov:
Departamento de Matematica, Universidade de Sao Paulo, Caixa Postal 66281, 05315-970 São Paulo, Brazil}
\email{shuragri@gmail.com}

\begin{document}

\maketitle

\section*{Introduction}

The notion of an algebraic supergroup is a generalization of the notion of an algebraic group. 
It is natural to ask which properties of algebraic groups remain valid also for algebraic supergroups.
For example, it is well known that if the ground field $K$ has characteristic zero
and $G$ is a connected algebraic group, then the Lie algebra $\mathsf{Lie}(G')$ of the commutant $G'$ of $G$ coincides
with the commutant $\mathsf{Lie}(G)'$ of $\mathsf{Lie}(G)$. Using the technique of Harish-Chandra superpairs,
recently developed in \cite{carfior, mas, vish}, we find a rather simple counterexample to analogous statement for algebraic supergroups (see Section 10 below). 
In other words, the equality $\mathsf{Lie}(G')=\mathsf{Lie}(G)'$ no longer holds in the category of algebraic supergroups!

Another interesting question is how are properties of an algebraic supergroup $G$ related to the same properties of its largest even subgroup $G_{ev}$.  
The second author proved that an algebraic supergroup $G$ is unipotent if and only if $G_{ev}$ is. The same result
has been reproved in \cite{mas} using the technique of Harish-Chandra superpairs. 
Other result of this kind proved in \cite{mas} states that a supergroup $G$ is simply connected if and only if $G_{ev}$ is.

Recall that every (even affine) supergroup $G$ has the largest normal unipotent supersubgroup $G_u$ that is called the unipotent radical of $G$, (cf. \cite{mas, zubul}). 
Keeping in mind the properties of algebraic groups, one can formulate the following questions.

Suppose that $H$ is a normal connected supersubgroup of a connected algebraic supergroup $G$. Is it true that $H_u=G_u\bigcap H$? 
Analogous property is valid in the category of algebraic groups (however, if the characteristic of $K$ is positive, then one has to assume additionally that both groups 
$G$ and $H$ are reduced). 

Consider the relationship of $(G_{ev})_u$ and $(G_u)_{ev}$. 
It is easy to see that if $(G_{ev})_u =1$, then $G_{u}$ is a finite odd supergroup. 
Next question we ask is whether $G_u =1$ implies that $(G_{ev})_u$ is finite?

Surprisingly, answers to both questions above are again negative. In Section 8 we construct a semi-direct product
of two abelian supergroups $H=X\rtimes G$ such that $H_u=1$ but $X_u\neq 1$. Additionally, $(H_{ev})_u$ is non-trivial connected supergroup.
  
We will use the same notation as in the classical setting and call a supergroup $G$ reductive if $G_u=1$.  
For the qround field of characteristic zero, a concept of quasi-reductive supergroup was introduced in \cite{serg}. 
Namely, $G$ is called quasi-reductive whenever $G_{ev}$ is reductive, or equivalently, $G_{ev}$ is linearly reductive. 
Note that $G_{ev}$ is reductive if and only if $\mathsf{Lie}(G_{ev})=\mathsf{Lie}(G)_0$ is a reductive Lie algebra. 
The Lie superalgebras of quasi-reductive supergroups were described in \cite{serg}. 
Any such Lie superalgebra $\mathfrak{L}$ has a filtration by superideals
$Z(\mathfrak{L})\subseteq\mathfrak{C}\subseteq\mathfrak{L}$, where $Z(\mathfrak{L})$ is the center of $\mathfrak{L}$,
$\mathfrak{C}/Z(\mathfrak{L})$ is a direct sum of minimal superideals and $\mathfrak{L}/\mathfrak{C}$ is semi-direct product
of its even part and an odd abelian superideal. 

However, the above nice description is not entirely satisfactory since we do not know how this filtration of $\mathfrak{L}$ corresponds to the subgroup structure of a related supergroup $G$. In other words, we would like to know if there are normal supersubgroups $Z$ and $C$ of $G$ 
such that $\mathsf{Lie}(Z)=Z(\mathfrak{L})$ and $\mathsf{Lie}(C)=\mathfrak{C}$. We prove that $Z$ is just the center $Z(G)$ of $G$ but the problem whether the superideal $\mathfrak{C}$ is algebraic remains unsolved.

Therefore we decided to develop a different approach working with supergroups. 
First we prove that if $\mathsf{Lie}(G)$ is semisimple, then there is a pair of supergroups $(U, H)$ (called a sandwich pair) such that $U\unlhd H$ and $U\leq G\leq H$ 
(see Theorem \ref{sandwich}). Moreover, 
$\mathsf{Lie}(U)=\mathfrak{U}$, $\mathsf{Lie}(H)=Der(\mathfrak{U})$ and the embedding
of $\mathfrak{G}=\mathsf{Lie}(G)$ into $Der(\mathfrak{U})$ satisfies conditions of Theorem 6 of \cite{kac}.

Afterward, we work with the solvable radical $R(G)$ of $G$. Then the Lie superalgebra of $\tilde{G}=G/R(G)$ is semisimple, and therefore
$\tilde{G}$ can be inserted as a middle term into a sandwich pair $(U, H)$. Since $\tilde{G}$ is also quasi-reductive,  
$U$ is quasi-isomorphic to a direct product of its normal supersubgroups $U_i$, where each $\mathsf{Lie}(U_i)$
belongs to the same family of minimal superideals as in Theorem 6.9 of \cite{serg}.
Moreover, $R(G)$ and $\tilde{G}/U$ are triangulizable supergroups (for the definition, see the end of Section 6) and their unipotent radicals are odd unipotent. 
Conversely, if $G$ has the above filtration by normal supersubgroups, then $G$ is quasi-reductive.

At the final step we observe that if $G$ is reductive and $G$ has no central toruses, then $\mathsf{Lie}(G)$ is semisimple. 
This means that, in principle, reductive supergroups can be described using sandwich pairs.

The paper is organized as follows. In the first six sections we give all necessary definitions, notations and derive auxiliary 
results. The most important results in these sections are Lemma \ref{AdcommuteswithLieoperation} and properties of the action of an affine supergroup on another affine supergroup by supergroup automorphisms. In the seventh section of this paper we completely describe such actions on an abelian supergroup. 
Based on that, in the eight section, we construct a semi-direct product $H=X\rtimes G$ such that $H_u =1$ but both $X_u$ and $(H_{ev})_u$ contain an one-dimensional unipotent subgroup.  The ninth section is devoted to a simple criterion describing when a supersubalgebra of a given algebraic Lie superalgebra is also algebraic. 
This is accomplished using the technique of Harish-Chandra superpairs. As a by-product, we describe when $\mathsf{Lie}(G')$ coincides with $\mathsf{Lie}(G)'$. 
In the tenth section we construct a supergroup $G$ such that $\mathsf{Lie}(G')\neq \mathsf{Lie}(G)'$ and give a simplest example of $G$ with this property.
The last section is devoted to properties of reductive and quasi-reductive supergroups described earlier in this introduction.

\section{Affine and algebraic supergroups}

In this section we follow definitions and notations from \cite{jan, maszub, zub}. 

Let $K$ be a field of characteristic $char K\neq 2$.
A \emph{vector superspace} is a vector space $V$ (over $K$) graded by the group $\mathbb{Z}_2 = \{0, 1\}$. The homogeneous components of $V$ are denoted by $V_0$, $V_1$. The degree (or {\it parity}) of a homogeneous element $v$ is denoted by $|v|$. If $V$ and $W$ are superspaces, then $\mathsf{Hom}_K(V, W)$ has the natural superspace structure defined by 
$$\mathsf{Hom}_K(V, W)_i=\{\phi| \phi(V_j)\subseteq W_{i+j}, i, j\in\mathbb{Z}_2\} .$$
Let $\mathsf{SMod}_K$ denote the $K$-linear abelian category of vector superspaces with even morphisms.
This is a tensor category with the canonical symmetry
$$ t=t_{V, W} :V \otimes W \overset{\simeq}{\longrightarrow} W \otimes V, \quad v \otimes w \mapsto (-1)^{|v||w|} w \otimes v,$$
where $V, W \in \mathsf{SMod}_K$. 

Classical definitions of algebraic objects can be extended to this symmetric tensor category using the adjective `super'. For example,
a \emph{(Hopf) superalgebra} is a (Hopf) algebra object in $\mathsf{SMod}_K$, and analogously for other definitions.  
Additionally, all superalgebras are assumed to be unital.  A superalgebra $A$ is called
{\it supercommutative}, if $ab=(-1)^{|a||b|}ba$ for all homogeneous elements $a, b\in A$.
Denote by $\mathsf{SAlg}_K$ the category of supercommutative superalgebras.

Any functor from $\mathsf{SAlg}_K$ to the category of sets is called a $K$-functor. For example, if
$V$ is a superspace, then one can define a $K$-functor by $V_a(A)=V\otimes A$ for $A\in\mathsf{SAlg}_K$.

A $K$-functor $X$ is said to be an {\it affine superscheme}, if $X$ is represented by a superalgebra
$A\in\mathsf{SAlg}_K$. In other words, $X(B)=\mathsf{Hom}_{\mathsf{SAlg}_K}(A, B)$ for $B\in\mathsf{SAlg}_K$. 
In notations from \cite{maszub, zub}, this functor $X$ is denoted by $SSp \ A$. 
A $K$-functor morphism $SSp \ A\to SSp \ B$ is (uniquely) defined by the dual superalgebra morphism $\phi^* : B\to A$.  

A closed subfunctor $Y$ of 
$X=SSp \ A$ is defined by a superideal $I_Y =I$ of $A$ such that $Y(B)=\{x\in X(B)| x(I)=0\}$. 
Thus $Y=V(I)\simeq SSp \ A/I$ is again an affine superscheme.
An open subfunctor $Y$ of $X$ is also defined by a superideal $I$ as $Y(B)=\{x\in X(B) | Bx(I)=B\}$.
For example, any $f\in A_0$ defines an open subfunctor $X_f$ that corresponds to the ideal $Af$.

For a $K$-functor $X$, define a subfunctor $X_{ev}(A)=X(\iota_0)X(A_0)$ for $A\in \mathsf{SAlg}_K$, where $\iota_0$ is the natural algebra embedding $A_0\to A$. 
For example, $(SSp \ R)_{ev}$ is a closed supersubscheme of $SSp \ R$, defined by the ideal $RR_1$.  

An affine superscheme $G=SSp \ A$ is a group $K$-functor if and only if $A$ is a Hopf superalgebra. If it is the case, then $G$ is called an {\it affine supergroup}. 
Besides, if $A=K[G]$ is finitely generated, then $G$ is called an {\it algebraic supergroup}. 
Denote by $\epsilon_G, \Delta_G$ and $s_G$ the counit, comultiplication and antipode of $K[G]$, respectively.
We use Sweedler notation, that is for every $f\in A$ we write
$$\Delta_G(f)=\sum f_1\otimes f_2,$$ 
$$(id_A\otimes\Delta_G)\Delta_G(f)=(\Delta_G\otimes id_A)\Delta_G(f)=\sum
f_1\otimes f_2\otimes f_3,$$
and so on.
The augmentation superideal $\ker\epsilon_G$ is denoted by $K[G]^+$. Moreover, for any supersubspace $V\subseteq K[G]$
denote $V\bigcap K[G]^+$ by $V^+$. 

\begin{exam}\label{generallinear}
Let $V$ be a vector superspace of superdimension $m|n$, that is, $\dim V_0=m, \dim V_1=n$. Denote by $GL(V)$, or
by $GL(m|n)$, the corresponding general linear supergroup which is a group $K$-functor such that for any $A\in\mathsf{SAlg}_K$ the group $GL(V)(A)$ consists
of all even and $A$-linear automorphisms of $V\otimes A$. 
Denote the generic matrix $(c_{ij})_{1\leq i, j\leq m+n}$ by $C$ and $d=\det(C_{00})\det(C_{11})$, where
$C_{00}=(c_{ij})_{1\leq i, j\leq m}$ and $C_{11}=(c_{ij})_{m+1\leq i, j\leq m+n}$ are even blocks of $C$.
It is easy to see that $K[GL(m|n)]=K[c_{ij}|1\leq i, j\leq m+n]_{d}$, where $|c_{ij}|=0$ if and only if $1\leq i, j\leq m$ or $m+1\leq i, j\leq m+n$.

The Hopf superalgebra structure of $K[GL(m|n)]$ is defined via
$$\Delta_{GL(m|n)}(c_{ij})=\sum_{1\leq k\leq m+n} c_{ik}\otimes c_{kj}, \ \epsilon_{GL(m|n)}(c_{ij})=\delta_{ij}, 1\leq i, j\leq m+n.$$
\end{exam}
\begin{exam}
Recall the definitions of (one-dimensional) torus $G_m$ and (one-dimensional connected) unipotent group 
$G_a$.

If $G=G_m$, then $K[G]=K[t^{\pm 1}], \Delta_G(t)=t\otimes t, \epsilon_G(t)=1$ and $s_G(t)=t^{-1}$.

If $G=G_a$, then $K[G]=K[t], \Delta_G(t)=t\otimes 1 +1\otimes t, \epsilon_G(t)=0$ and $s_G(t)=-t$. 
In both cases $|t|=0$.
\end{exam}
\begin{exam}
A supergroup $G$ is called (one-dimensional) {\it odd unipotent}, and denoted by $G_a^-$, if $K[G]=K[t], |t|=1, 
\Delta_G(t)=t\otimes 1 +1\otimes t, \epsilon_G(t)=0$ and $s_G(t)=-t$.
\end{exam}
A closed supersubscheme $H$ of $G$ is a subgroup functor if and only if $I_H$ is a Hopf superideal of $K[G]$. In what follows, all supersubgroups are assumed to be closed unless otherwise stated. If $H$ is a supersubgroup of $G$, we denote it by $H\leq G$. For example, $G_{ev}\leq G$ and $G_{ev}$ is called the {\it largest even} supersubgroup of $G$. 

If $H(A)$ is a normal subgroup of $G(A)$ (we denote this by $H(A)\unlhd G(A)$) for every $A\in\mathsf{SAlg}_K$, then $H$ is called a {\it normal} supersubgroup of $G$ and we write  
$H\unlhd G$. 
The center of $G$ is a $K$-subgroup functor $Z(G)$ such that for every $A\in\mathsf{SAlg}_K$ :
$$Z(G)(A)=\{g\in G(A) |\forall A'\in\mathsf{SAlg}_K, \forall \phi : A\to A', G(\phi)(g)\in Z(G(A'))\}.$$
One can show that $Z(G)$ is a supersubgroup (see \cite{zubul}) and it is clear that $Z(H)\unlhd G$ whenever $H\unlhd G$. 

If $\phi : G\to H$ is a supergroup morphism, then $\ker\phi$ is a normal supersubgroup of $G$ that is defined by the Hopf superideal $K[G]\phi^*(K[H]^+)$.

\section{Superalgebras of distributions and Lie superalgebras}

Let $G$ be an algebraic supergroup. Denote $K[G]^+$ by $\mathfrak{m}$. A Hopf superalgebra of distributions $\mathsf{Dist}(G)$ coincides with $\bigcup_{l\geq 0}\mathsf{Dist}_l(G)$, where each component $\mathsf{Dist}_l(G)=(K[G]/\mathfrak{m}^{l+1})^*$ is a supersubcoalgebra dual to 
a superalgebra $K[G]/\mathfrak{m}^{l+1}$ (cf. \cite{jan, zub}).  
The superalgebra structure of $\mathsf{Dist}(G)$ is given by the convolution
$$\phi\star\psi(f)=\sum (-1)^{|\psi||f_1|}\phi(f_1)\psi(f_2), f\in K[G].$$
More generally, for any $A\in\mathsf{SAlg}_K$ the superspace $\mathsf{Hom}_K(K[G], A)$ is an unital superalgebra with respect to the same convolution and unit $\epsilon_G$.

The supersubspace $\mathsf{Dist}_1^+$ has a Lie superalgebra structure with respect to the operation 
$[\phi, \psi]=\phi\star\psi-(-1)^{|\phi||\psi|}\psi\star\phi$. It is called a {\it Lie superalgebra} of $G$ and denoted by
$\mathsf{Lie}(G)$. If $char K=0$, then $\mathsf{Dist}(G)$ coincides with the {\it universal enveloping} superalgebra $\mathsf{U}(\mathsf{Lie}(G))$ of $\mathsf{Lie}(G)$.
\begin{exam}
$\mathsf{Lie}(GL(m|n))=\mathfrak{gl}(m|n)$ is the general linear Lie superalgebra (cf. \cite{frsorsci}).
\end{exam}
The $K$-functor $\mathsf{Lie}(G)_a$ is isomorphic to the {\it Lie superalgebra functor} ${\bf Lie}(G)$ of $G$ (see Lemma 3.2 of \cite{zub}) defined as follows.
For $A\in\mathsf{SAlg}_K$, let $A[\epsilon_0,  \epsilon_1]$ be a (supercommutative) superalgebra of dual numbers.
By definition, $A[\epsilon_0,  \epsilon_1] = \{a+\epsilon_0 b+\epsilon_1 c |  a, b, c\in A\}$,  
$\epsilon_i\epsilon_j = 0, |\epsilon_i|=i$, where $i, j\in\{0, 1\}$.
There are two morphisms of superalgebras $p_A : A[\epsilon_0,  \epsilon_1]\to A$ and $i_A : A\to A[\epsilon_0,  \epsilon_1]$
defined by $a+\epsilon_0 b+\epsilon_1 c\mapsto a$ and a $a\mapsto a$, respectively. The functor ${\bf Lie}(G)$ is defined 
as ${\bf Lie}(G)(A) = \ker(G(p_A))$ for $A\in\mathsf{SAlg}_K$. 

The isomorphism $\mathsf{Lie}(G)_a\simeq{\bf Lie}(G)$ is given by
$$(v\otimes a)(f) = \epsilon_G(f) + (-1)^{|a||f|}\epsilon_{|v\otimes a|}v(f)a,$$
where $v\in \mathsf{Lie}(G) = (\mathfrak{m}/\mathfrak{m}^2)^* $, 
$a\in A$ and $f\in K[G].$

If we identify $\mathsf{Lie}(G)\otimes A$ with $\mathsf{Hom}_K(\mathfrak{m}/\mathfrak{m}^2 , A)$ via $(v\otimes a)(f)
=(-1)^{|a||f|}v(f)a$, then the above isomorphism can be represented as
$$u\mapsto \epsilon_G + \epsilon_0 u_0 + \epsilon_1 u_1$$
for  $u\in\mathsf{Hom}_K(\mathfrak{m}/\mathfrak{m}^2 , A).$

The supergroup $G$ acts on the functor ${\bf Lie}(G)$ by
$$(g, x)\mapsto G(i_A)(g)xG(i_A)(g)^{-1}$$ 
for $g\in G(A)$, $x\in {\bf Lie}(G)(A)$ and $A\in\mathsf{SAlg}_K.$
This action is called the {\it adjoint} action and denoted by $\mathsf{Ad}$. 
According to Lemma 3.2 of \cite{zub}, $\mathsf{Ad} : G\to GL(\mathsf{Lie}(G))$ is a morphism of supergroups.

Let $\mathfrak{L}$ be a Lie superalgebra. For any $A\in\mathsf{SAlg}_K$ one can define a Lie superalgebra structure on $\mathfrak{L}\otimes A$
by the rule
$$[x\otimes a, y\otimes b]=(-1)^{|a||y|}[x, y]\otimes ab$$
for $x, y\in \mathfrak{L}$ and $a, b\in A.$
Therefore $\mathfrak{L}_a$ is a functor from $\mathsf{SAlg}_K$ to the category of Lie superalgebras.

As above, the (super)commutator $[\phi, \psi]=\phi\star\psi-(-1)^{|\phi||\psi|}\psi\star\phi$ defines a Lie superalgebra structure on $\mathsf{Hom}_K(K[G], A)$.  
\begin{lemma}\label{convolution}
The superspace $\mathsf{Hom}_K(\mathfrak{m}/\mathfrak{m}^2 , A)$ is a Lie supersubalgebra of
$\mathsf{Hom}_K(K[G], A)$ that is isomorphic to $\mathsf{Lie}(G)\otimes A$.
\end{lemma}
\begin{proof}
The proof follows from straightforward calculations.
\end{proof}

For the later use, consider the following operation.
$$(\mathsf{Ad}(g)u)(f)=\sum (-1)^{|u||f_1|}g(f_1)u(f_2)g(s_G(f_3))$$
for $u\in\mathsf{Hom}_K(\mathfrak{m}/\mathfrak{m}^2, A)$,
$g\in G(A)$ and $f\in\mathfrak{m}.$ 
\begin{lemma}\label{AdcommuteswithLieoperation}
The adjoint action of $G$ on $\mathsf{Lie}(G)_a\simeq {\bf Lie}(G)$ commutes with the above operation. 
\end{lemma}
\begin{proof}
For any $u, v\in\mathsf{Hom}_K(\mathfrak{m}/\mathfrak{m}^2, g\in G(A)$ and $f\in\mathfrak{m}$  we have
$$[\mathsf{Ad}(g)u, \mathsf{Ad}(g)v](f)=$$
$$\sum(-1)^{|v|(|f_1|+|f_2|+|f_3|)+|u||f_1|+|v||f_4|}g(f_1)u(f_2)g(s_G(f_3))g(f_4)v(f_5)g(s_G(f_6))-
$$
$$
\sum(-1)^{|u||v|+|u|(|f_1|+|f_2|+|f_3|)+|v||f_1|+|u||f_4|}g(f_1)v(f_2)g(s_G(f_3))g(f_4)u(f_5)g(s_G(f_6))=
$$
$$
\sum(-1)^{|v|(|f_1|+|f_2|+|f_3|)+|u||f_1|}g(f_1)u(f_2)g(\epsilon_G(f_3))v(f_4)g(s_G(f_5))-
$$
$$
\sum(-1)^{|u||v|+|u|(|f_1|+|f_2|+|f_3|)+|v||f_1|}g(f_1)v(f_2)g(\epsilon_G(f_3))u(f_4)g(s_G(f_5))=
$$
$$
\sum(-1)^{|v|(|f_1|+|f_2|)+|u||f_1|}g(f_1)u(f_2)v(f_3)g(s_G(f_4))-
$$
$$
\sum(-1)^{|u||v|+|u|(|f_1|+|f_2|)+|v||f_1|}g(f_1)v(f_2)u(f_3)g(s_G(f_4))=
(\mathsf{Ad}(g)[u, v])(f).
$$
\end{proof}
\section{Connected supergroups}

An affine superscheme $SSp \ A$ is called {\it connected} if the underlying topological space $|SSpec \ A|^e$ of
the affine superspace $SSpec \ A$ is connected. Since $|SSpec \ A|^e=|Spec \ A_0|^e$, $SSp \ A$ is connected if and only if $A_0$ has no nontrivial idempotents (cf. 5.5. of \cite{wat}).

Suppose that $A$ is finitely generated. Then $A_0$ is finitely generated and it contains the largest (finite-dimensional) separable subalgebra $\pi_0(A_0)$ (cf. 6.7 of \cite{wat}). We denote $\pi_0(A_0)$ also by $\pi_0(A)$. 
\begin{lemma}\label{defofseparab}
$\pi_0(A)$ is the largest separable subalgebra in $A$.
\end{lemma}
\begin{proof}
Let $B$ be a separable subalgebra of $A$. According to Theorem 6.2(2) of \cite{wat}, $B\otimes\overline{K}$ is a separable subalgebra
of $A\otimes\overline{K}$. Moreover, $B\otimes\overline{K}=\overline{K}^t$ is generated by pairwise orthogonal idempotents. Using Lemma 7.2 of \cite{zub} we infer
that $B\otimes\overline{K}\subseteq A_0\otimes\overline{K}$. Thus $B\subseteq A_0$.
\end{proof}
\begin{lemma}\label{propofpi}
If $A$ and $B$ are finitely generated commutative superalgebras, then $\pi_0(A\otimes B)=\pi_0(A)\otimes\pi_0(B)$. 
\end{lemma} 
\begin{proof}
Observe that $(A\otimes B)_0=(A_0\otimes B_0)\bigoplus (A_1\otimes B_1)$ and $(A_1^2\otimes B_0 +A_0\otimes B_1^2)\bigoplus (A_1\otimes B_1)$ is a nilpotent ideal in $(A\otimes B)_0$.
By Lemma 6.8 and Theorem 6.5 of \cite{wat}, $\pi_0((A\otimes B)_0)\simeq\pi_0(A_0/A_1^2\otimes B_0/B_1^2)\simeq \pi_0(A)\otimes\pi_0(B)$.
\end{proof}
Let $G$ be an algebraic supergroup. Using the same arguments as in 6.7 of \cite{wat} we derive that $B=\pi_0(K[G])$ is an (even) Hopf subalgebra 
of $K[G]$. There is a natural epimorphism of supergroups $G\to \pi_0 G=SSp \ B$. 
It is clear that $\pi_0 G$ is the largest etale factor-supergroup of $G$. In other words, any epimorphism $G\to H$, where
$H$ is an (even) etale supergroup, factors through $G\to\pi_0 G$. The kernel of $G\to \pi_0 G$ is called
a {\it connected component} of $G$ and is denoted by $G^0$. 

The algebra $B$ is isomorphic to $L_1\times\ldots\times L_t$, where each $L_i$ is a separable field extension of $K$.
Choose (pairwise orthogonal) idempotents $e_1, \ldots, e_t$ such that $L_i=Be_i$ for $1\leq i\leq t$. One can assume that $\epsilon(e_1)=1$ and $\epsilon(e_i)=0$ for $i > 1$. 
It is clear that $B^+ =B(1-e_1)$. Thus $I_{G^0}=K[G](1-e_1), K[G^0]\simeq K[G]e_1$ and $G^0=G_{e_1}$ is an open subfunctor of $G$ (see Remark 9.3 of \cite{zub}).

Summarizing, $G$ is connected if and only if $G_{ev}$ is connected if and only if $G=G^0$ (see Theorem 6.7 of \cite{wat}). 

Recall from \cite{zub} that a {\it pseudoconnected component} $G^{[0]}$ of $G$ is a normal supersubgroup defined by the Hopf superideal
$\bigcap_{n\geq 1} \mathfrak{m}^n$.  
\begin{lemma}\label{inclusion}
Let $G$ be an algebraic supergroup. Then
$G^{[0]}\leq G^0$.
\end{lemma}
\begin{proof}
By Lemma 9.2 of \cite{zub}, $f$ belongs to $I_{G^{[0]}}$ if and only if there exists $g\in\mathfrak{m}_0$ such that $(1-g)f=0$.
Since $e_1 I_{G^0} =0$ and $1-e_1\in\mathfrak{m}_0$, we obtain that $I_{G^0}\subseteq I_{G^{[0]}}$. 
\end{proof}
Lemma 9.6 of \cite{zub} states that if $char K=0$, then $G^0=G^{[0]}$.
The following Proposition answers Question 9.2 of \cite{zub} and shows that the same statement is valid also in the case $char K >0$. 
\begin{prop}\label{pseudconcriterion}
If $G$ is an algebraic supergroup $G$, then  $G^0=G^{[0]}$.
\end{prop}
\begin{proof}
By Lemma \ref{inclusion}, one can assume that $G$ is connected.
By Theorem 4.5 of \cite{mas1} we have $A\simeq \overline{A}\otimes\Lambda(W)$, where $A=K[G], \overline{A}=K[G_{ev}]$ and
$W=A_1/A^+_0A_1$. 
It is clear that $\mathfrak{m}$ can be identified with $\overline{\mathfrak{m}}\otimes\Lambda(W) +\overline{A}\otimes\Lambda(W)^+$, where $\Lambda(W)^+ =\Lambda(W)W$. 
Using arguments from the proof of Lemma 9.1 of \cite{zub} we obtain 
$\bigcap_{n\geq 0}\mathfrak{m}^n\subseteq \bigcap_{n\geq 0}\overline{\mathfrak{m}}^n\otimes\Lambda(W)$. On the other hand,
each $f\in \bigcap_{n\geq 0}\overline{\mathfrak{m}}^n\otimes\Lambda(W)$ is annihilated by an element $(1-y)\otimes 1_{\Lambda(W)}$, where $y\in \overline{\mathfrak{m}}$. 
Therefore $I_{G^{[0]}}=I_{G_{ev}^{[0]}}\otimes\Lambda(W)$. Since $G_{ev}$ is connected, Exercise 6 from \cite{wat} implies $I_{G_{ev}^{[0]}}=0$. 
\end{proof}
\begin{rem}\label{incorrect}
We would like to clarify that the statement of Proposition 9.2 in \cite{zub} is valid only when $char K = 0$ and it does not hold when $char K>0$. 
For example, for any algebraic (super)group $G$, its $n$-th infinitesimal (super)subgroup $G_n$, where $n\geq 1$, satisfies $\mathsf{Lie}(G)=\mathsf{Lie}(G_n)$, but $G/G_n$ is not necessary finite.   
Additionally, in both statements of Lemma 9.6 one needs to assume that $char K=0$. This assumption was accidentaly lost in the final version of the article \cite{zub}.
\end{rem}

\section{Quotients}

Let $G$ be an algebraic supergroup and $H\leq G$. The sheafification of the $K$-functor $A\to (G/H)_{(n)}(A)=G(A)/H(A), A\in \mathsf{SAlg}_K,$ is called a \emph{sheaf quotient} and is denoted by
$G / H$.  The functor $(G/H)_{(n)}$ is called the {\it naive quotient}. It was proved in \cite{maszub} that a quotient sheaf $G / H$ is a Noetherian superscheme and the quotient morphism $\pi : G\to X$ is affine and faithfully flat. 

If $H\unlhd G$, then $G/ H\simeq SSp \ K[G]^H$ is an algebraic supergroup. Let $L$ be a supersubgroup of $G$ and $I_L$ be its
defining Hopf superideal.
A sheafification of the group subfunctor $A\to L(A)H(A)/H(A), A\in \mathsf{SAlg}_K,$ in $G/ H$, is denoted by $\pi(L)$.
It is a supersubgroup of $G/ H$ defined by the Hopf superideal $K[G]^H\bigcap I_L$ (cf. Theorem 6.1 of \cite{zub}). 
The supersubgroup $\pi^{-1}(\pi(L))$ is denoted by $LH$. As it was observed on p. 735 of \cite{zub}, $LH$ is a sheafification of
the group subfunctor $A\to L(A)H(A), A\in \mathsf{SAlg}_K$ and its defining Hopf superideal is $K[G](K[G]^H\bigcap I_L)$. 

Notice that if $L\bigcap H=1$, then $LH$ can be identified with the direct product $L\times H\leq G$, where
$(L\times H)(A)=L(A)\times H(A)\leq G(A), A\in\mathsf{SAlg}_K$. In fact, $L$ is identified with a subfunctor
of $(G/H)_{(n)}$. Since $L$ is affine, its sheafification is equal to itself (cf. \cite{zub, jan}).
\begin{lemma}\label{productofnormalandsome}
If  $char K=0$, then $\mathsf{Lie}(LH)=\mathsf{Lie}(L)+\mathsf{Lie}(H)$.
\end{lemma} 
\begin{proof}
The proof follows from Corollary 6.1 and Proposition 9.1 of \cite{zub}.
\end{proof}
\begin{lemma}\label{factorofconnect}
If $G$ is a connected algebraic supergroup and $H$ is a normal supersubgroup of $G$, then $G/H$ is connected. 
\end{lemma}
\begin{proof}
The statement follows from $\pi_0(K[G]^H)\subseteq\pi_0(K[G])=K$. 
\end{proof}
\begin{lemma}\label{charsubgr} 
Let $G$ be an algebraic supergroup. Then we have the following.
\begin{enumerate}
\item If $H$ is a supersubgroup of $G$, then $H^0\leq H\bigcap G^0$;
\item If $H\unlhd G$, then $H^0\unlhd G$;
\item $(G_{ev})^0=(G^0)_{ev}$. 
\end{enumerate}
\end{lemma}
\begin{proof}
$H/(H\bigcap G^0 )$ is isomorphic to a (super)subgroup of $G/G^0$ (see Theorem 6.1 of \cite{zub} and Lemma 3.3 of \cite{zub1}). According to Corollary 6.2 of \cite{wat}, 
$H/(H\bigcap G^0 )$ is an etale (super)group and the first statement follows.

Since $H^0=H^{[0]}$, the second statement follows by remark after Lemma 9.1 in \cite{zub}. 

The superalgebra epimorphism $A\to \overline{A}=A/AA_1=K[G_{ev}]$, restricted on $B=\pi_0(A)$, is an isomorphism onto $\pi_0(\overline{A})$. This implies the third statement.    
\end{proof}

\section{Some auxiliary properties of algebraic supergroups}

\begin{lemma}\label{normal}
Let $char K=0$ and $G$ be a supersubgroup of $H$. If both $G$ and $H$ are connected and $\mathfrak{G}=\mathsf{Lie}(G)$ is a superideal of $\mathfrak{H}=\mathsf{Lie}(H)$, 
then $G\unlhd H$.
\end{lemma}
\begin{proof}
According to Lemma 9.5 of \cite{zub}, $V^G=V^{\mathfrak{G}}$ for any $G$-supermodule $V$. 
The statement then follows from Lemma 9.7 of \cite{zub}. 
\end{proof}
\begin{lemma}\label{whenLiesuperalgisabelian}
Assume a supergroup $G$ is connected. Then $G$ is abelian if and only if $\mathsf{Dist}(G)$ is a supercommutative superalgebra. 
If $char K=0$, then $G$ is abelian if and only if its Lie superalgebra is abelian.
\end{lemma}
\begin{proof}
The supergroup $G$ is commutative if and only if $K[G]$ is supercocommutative, that is 
\[\Delta(f)=\sum f_1\otimes f_2=\sum(-1)^{|f_1||f_2|}f_2\otimes f_1\] for every $f\in K[G]$.
Since $\bigcap_{t\geq 0}\mathfrak{m}^t =0$, $K[G]$ is supercocommutative if and only if any element
$\phi\otimes\psi\in\mathsf{Dist}(G)^{\otimes 2}$ vanishes on $\sum f_1\otimes f_2-\sum(-1)^{|f_1||f_2|}f_2\otimes f_1$.
This proves the first statement. 
The second statement follows from Lemma 3.1 of \cite{zub}.
\end{proof}
\begin{lemma}\label{centralsubgroup}
Let $char K=0$, $G$ be a connected algebraic supergroup and $H$ be its connected supersubgroup such that $\mathsf{Lie}(H)$ is central in
$\mathsf{Lie}(G)$. Then $H\leq Z(G)^0$.
\end{lemma}
\begin{proof}
Denote $\mathfrak{G}=\mathsf{Lie}(G)$ and $\mathfrak{H}=\mathsf{Lie}(H)$.
The conjugation action $G\times G\to G$ is defined by a superalgebra morphism 
$$\nu : f\mapsto \sum (-1)^{|f_1||f_2|} f_2\otimes f_1 s_G(f_3)$$ for $f\in K[G].$
Since $G$ and $H$ are connected, the condition $\nu(f)-f\otimes 1\in I_H\otimes K[G]$ is equivalent to
\[\sum (-1)^{|\phi_1||\psi|}s_{\mathsf{U}(\mathfrak{G})}(\phi_1)\psi\phi_2=\epsilon_{\mathsf{U}(\mathfrak{G})}(\phi)\psi\]
for $\phi\in\mathsf{U}(\mathfrak{G})$ and $\psi\in \mathsf{U}(\mathfrak{H}).$
Since $\mathsf{U}(\mathfrak{H})\subseteq Z(\mathsf{U}(\mathfrak{G}))$, the last condition is satisfied automatically. Thus $H\leq Z(G)^0$.
\end{proof}

Let $G$ be an algebraic supergroup. A {\it commutator supersubgroup} or {\it commutant} of $G$ is the smallest closed supersubgroup of $G$, that is denoted by $G'$, such that
$G'(A)$ contains all products
of commutators \[[x_1, y_1]\ldots [x_n, y_n],\]
where $x_i, y_i\in G(A), A\in \mathsf{SAlg}_K$ and $n\geq 1$.
Inductively, define a solvable series of $G$ as $G^{(1)}=G', G^{(k+1)}=(G^{(k)})'$.
Then all supersubgroups $G^{(k)}$ are normal, and if $G$ is connected, then they are also connected (cf. \cite{zubul}).   
A supergroup $G$ is called {\it solvable} if $G^{(s)}=1$ for some positive integer $s$.

An algebraic supergroup $G$ is called {\it unipotent} if every simple $G$-supermodule is trivial (cf. \cite{zub1, mas, zubul}).  
The largest normal unipotent supersubgroup of $G$ is called the {\it unipotent radical} of $G$ (see \cite{mas, zubul}) and is denoted by $G_u$.
The supergroup $G$ is called {\it reductive} if $G_u =1$. 

The following proposition was proved in \cite{zubul}.
\begin{prop}\label{structureofabelian}
Suppose that $K$ is algebraically closed and $char K=0$. If $G$ is abelian, then $G\simeq G_s\times G_u$, where
$G_s$ is diagonalizable and $G_u\simeq G_a^k\times (G_a^-)^t$. If $G$ is connected, then $G_s$ is a torus.
\end{prop}
If $char K=0$, then $G$ is said to be {\it odd} if $\mathsf{Lie}(G)=\mathsf{Lie}(G)_1$. This is equivalent to $G\simeq (G_a^-)^k$ for some $k\geq 1$. 
The supergroup $G$ is called {\it triangulizable} if $G=G_u\rtimes T$, where $T$ is a torus. It is clear that every triangulizable $G$ is solvable.

\section{Actions}

Let $G$ and $X$ be affine supergroups and $G$ acts on $X$ via supergroup automorphisms.
In other words, we have a commutative diagram 
$$\begin{array}{ccccc}
X\times X\times G & \stackrel{m\times id_G}{\to} & X\times G & \stackrel{\rho}{\to} & X \\
\downarrow & & & & \uparrow\\
(X\times G)\times (X\times G) & & \stackrel{\rho\times\rho }{\to} & & X\times X   
\end{array}.
$$
Here $m : X\times X\to X$ is a multiplication morphism, $\rho : X\times G\to X$ is the action morphism,
the left vertical arrow is defined as  $(x, y, g)\mapsto (x, g; y, g)$ for $x, y\in X(A), g\in G(A)$ and $A\in \mathsf{SAlg}_K,$
and the right vertical arrow is given by $m$. 
The dual diagram is
$$\begin{array}{ccccc}
K[X] & \stackrel{\rho^*}{\to} & K[X]\otimes K[G] & \stackrel{\Delta_X\otimes id_{K[G]}}{\to} &
K[X]^{\otimes 2}\otimes K[G] \\
\downarrow & & & & \uparrow \\
K[X]^{\otimes 2} & & \stackrel{{\rho^*}^{\otimes 2}}{\to} & & (K[X]\otimes K[G])^{\otimes 2}
\end{array}.
$$
Here the right vertical arrow is defined as 
\[f_1\otimes h_1\otimes f_2\otimes h_2\mapsto (-1)^{|h_1||f_2|}f_1\otimes f_2\otimes h_1 h_2\]
for $f_1, f_2\in K[X]$ and $h_1, h_2\in K[G]$.
We call this diagram a {\it principal diagram}.

Since the action preserves the unit of $X$, we have $\sum \epsilon_X(f_1)h_2=\epsilon_X(f)$ for every $f\in K[X]$, where
$\rho^*(f)=\sum f_1\otimes h_2$. 

Observe that a closed supersubscheme $Y$ of $X$ is $G$-stable if and only if
$\rho^*(I_Y)\subseteq I_Y\otimes K[G]$, that is, $I_Y$ is a supersubcomodule of $K[X]$. Besides, 
$G$ acts on $Y$ trivially if and only if $\rho^*(f)-f\otimes 1\in I_Y\otimes K[G]$ for every $f\in K[X]$. 
\begin{exam}\label{torusaction}
Let $G=G_m$ acts on $X$. Then $B=K[X]$ has a $\mathsf{Z}$-grading $B=\oplus_{i\in \mathsf{Z}} B_i$, where $B_i=\{f\in B|\rho^*(f)=f\otimes t^i\}$ and
$B_i B_j\subseteq B_{i+j}$ for $i, j\in\mathsf{Z}$. These conditions are equivalent to 
$$\Delta_X(B_i)\subseteq \oplus_{k} (B_{i-k}\otimes B_k)$$
for $i\in\mathsf{Z},$ and $\oplus_{i\neq 0}B_i\subseteq B^+$.
\end{exam}

\begin{exam}\label{evenunipaction}
Let $G=G_a$ acts on $X$ and  $char K=0$. Then for each $f\in B=K[X]$ we have 
$$\rho^*(f)=\sum_{i\geq 0} f\delta_i\otimes t^i,$$ 
where the linear maps $\delta_i : B\to B$ satisfy 
$$C^k_i \delta_i=\delta_k\delta_{i-k}$$
for $0\leq k\leq i$, $\delta_0=id_B$, 
and $f\delta_i\in B^+$ whenever $i>0$. 
Besides, 
$$(*) \ (fg)\delta_i=\sum_{0\leq k\leq i}(f\delta_k) (g\delta_{i-k}),$$ 
$$(**) \ \Delta_X(f\delta_i)=
\sum_{0\leq k\leq i}\sum f_1\delta_k\otimes f_2\delta_{i-k}$$
for $i\geq 0,$ where $\Delta_X(f)=\sum f_1\otimes f_2$. 

In particular,  
$\delta_i=\frac{1}{i!}\delta_1^i$ for $i\geq 1$, $\delta_1^0=\delta_0$ and 
$\delta_1$ is an even (right) derivation. 
Moreover, the equation $(*)$ is nothing but
a Leibniz's rule and it is satisfied whenever $\rho$ is a $G$-action.
The equation $(**)$ can be rewritten as 
$$\Delta_X(f\delta_1^i)=\sum\sum_{0\leq k\leq i} C^k_i(f_1\delta_1^k)\otimes(f_2\delta_1^{i-k})$$
and can be regarded as {\it co-Leibniz's rule}. This equation is satisfied whenever $\delta_1$ is an even
{\it right coderivation}, that is $\delta_1$ satisfies $(**)$.  
\end{exam}
\begin{exam}\label{oddunipaction}
Let $G=G_a^-$ acts on $X$ and $B=K[X]$. 
Then $\rho^*(f)=f\otimes 1 + f\delta\otimes t$ for a linear map $\delta : B\to B^+$. As in the previous example, one can easily check that
$\delta$ is an odd superderivation and supercoderivation of $B$, that is
$$(ab)\delta=a(b\delta) +(-1)^{|b|}(a\delta)b,$$ 
$$\Delta_G(b\delta)=\sum b_1\otimes b_2\delta +
(-1)^{|b_2|} b_1\delta\otimes b_2$$
for $a, b\in B$.
\end{exam}
\begin{lemma}\label{rigidtheorem}
If $G$ is connected and $X$ is a diagonalizable (super)group, then $G$ acts on $X$ trivially.
\end{lemma} 
\begin{proof}
By definition, $B=K[X]$ is generated by group-like elements (cf. \cite{wat}).   
If $f$ is a group-like element, then $\rho^*(f)=\sum f_1\otimes h_2$, where all elements $f_1$ are group-like and linearly independent. Using the principal diagram 
we obtain that the elements $h_2$ are pairwise orthogonal idempotents and 
$\sum h_2=1$. Since $G$ is connected, only one $h_2$ is non-zero. Thus $h_2=1$ and $f_1=f$. 
\end{proof}
\begin{lemma}\label{primitiveelements}
The supersubspace $P(B)$ of all primitive elements of $B=K[X]$ is a supersubcomodule.
\end{lemma}
\begin{proof}
If $f\in B$ is primitive, then computing in the principal diagram gives 
$$\sum\Delta_X(f_1)\otimes h_2=\sum (f_1\otimes 1+1\otimes f_1)\otimes h_2.$$
If we choose $h_2$ to be linearly independent, then we infer that all $f_1$ are primitive. 
\end{proof}

\section{Actions on abelian supergroups}

If $X=G_a^l\times (G_a^-)^k$ is an abelian unipotent supergroup, then the superalgebra $K[X]$ is freely generated by
basic elements $z_1, \ldots , z_l, z_{l+1}, \ldots, z_{l+k}$ of $P(K[X])$. Here $|z_i|=0$ for $1\leq i\leq l$ and 
$|z_i|=1$ for $l+1, \ldots, l+k$. The superalgebra $K[X]$ has a basis consisting of monomials $z^{\lambda}=\prod_{1\leq i\leq l+k}z_i^{\lambda_i}$, where
$0\leq \lambda_i\leq 1$ whenever $i\geq l+1$. Denote $\sum_{1\leq i\leq l+k}\lambda_i$ by $|\lambda|$ and  define a partial order on the set of exponents by $\lambda\leq\mu$ if 
$\lambda_i\leq\mu_i$ for each $1\leq i\leq l+k$.  
Then
$$\Delta_X(z^{\lambda})=\sum_{\mu\leq\lambda} (-1)^{k_{\lambda, \mu}} C^{\mu}_{\lambda}z^{\mu}\otimes z^{\lambda-\mu},$$ 
where $C^{\mu}_{\lambda}=\prod_{1\leq i\leq l}C^{\mu_i}_{\lambda_i}$ and $k_{\lambda, \mu}$ is the number of pairs $(i, j)$ such that 
$l+1\leq j<i\leq l+k$ satisfying $\mu_i=1, \lambda_j-\mu_j=1$. 

Let $X=X_s\times X_u$, where $X_s$ is a diagonalizable (super)group and $X_u=G_a^l\times (G_a^-)^k$ is the unipotent radical of $X$. Denote the character group of $X_s$ by $D$. Then the elements of $D$ form a basis of $K[X_s]$.  

Let $G$ be a connected supergroup acting on $X$ via $\rho : X\times G\to X$.
Then by Lemma \ref{primitiveelements}, 
$$\rho^*(z_i)=\sum_{1\leq j\leq l+k} z_j\otimes f_{ji}, f_{ji}\in K[G]_{|z_i|+|z_j|}$$
for $1\leq i\leq l+k$.
\begin{theorem}\label{actiononabelian}
The supergroup $G$ acts trivially on $X_s$. If $char K=0$, then the action $\rho$ is uniquely defined by the collection of group homomorphisms $f_i$ from $D$ to the additive group $(K[G]K[G]_1)_{|z_i|}$ such that 
$$\Delta_G(f_i(g))=f_i(g)\otimes 1 +\sum_{1\leq j\leq l+k} f_{ij}\otimes f_j(g)$$
for $1\leq i\leq l+k$ and every $g\in D$. 
\end{theorem}
\begin{proof}
We can identify $K[X_s]$ and $K[X_u]$ with Hopf supersubalgebras of $K[X]$.
Then $P(K[X])=P(K[X_u])$, and $P(K[X_u])$ generates the Hopf superideal $I_{X_s}$ in $K[X]$. 
By Lemma \ref{primitiveelements}, $X_s$ is a $G$-stable, and by Lemma \ref{rigidtheorem}, $G$ acts trivially on $X_s$.
The last statement can be reformulated in the previous notations as 
$$\rho^*(g)=g\otimes 1 +\sum_{h\in D, \lambda >0} hz^{\lambda}\otimes f_{h, \lambda}(g)$$
for $g\in D$ and $f_{h, \lambda}(g)\in K[G]$, where $|f_{h, \lambda}(g)|=|z^{\lambda}|$. 
Using the principal diagram we infer 
$$\sum_{h\in D, \lambda>0} (h\otimes h)\Delta_{X_u}(z^{\lambda})\otimes f_{h, \lambda}(g)=
\sum_{h\in D, \lambda>0}hz^{\lambda}\otimes g\otimes f_{h, \lambda}(g) +$$
$$\sum_{h\in D, \lambda>0}g\otimes hz^{\lambda}\otimes f_{h, \lambda}(g) +
\sum_{h, h'\in D, \pi, \chi>0}(-1)^{|z^{\pi}||z^{\chi}|}hz^{\pi}\otimes h' z^{\chi}\otimes f_{h, \pi}(g)f_{h' , \chi}(g).
$$
Therefore $f_{h, \lambda}=0$ for any $h\neq g$. If we denote $f_{g, \lambda}(g)$ by $f_{\lambda}(g)$, then the previous formula for
$\Delta_{X_u}(z^{\lambda})$ implies
$$\sum_{\lambda>0}\sum_{0<\mu <\lambda}(-1)^{k_{\lambda, \mu}}C^{\mu}_{\lambda} gz^{\mu}\otimes gz^{\lambda-\mu}\otimes f_{\lambda}(g)=
\sum_{0<\pi, \chi} (-1)^{|z^{\pi}||z^{\chi}|} gz^{\pi}\otimes gz^{\chi}\otimes f_{\pi}(g)f_{\chi}(g).$$
In particular, if there is at least one number $i\geq l+1$ such that $\pi_i+\chi_i\geq 2$, then
$f_{\pi}(g)f_{\chi}(g)=0$. Otherwise, 
$$(-1)^{k_{\pi+\chi, \pi}}C^{\pi}_{\pi+\chi}f_{\pi+\chi}(g)=(-1)^{|z^{\pi}||z^{\chi}|}f_{\pi}(g)f_{\chi}(g).$$
If $char K=0$, then this recursive formula gives
$$f_{\lambda}(g)=\prod_{1\leq i\leq l+k}\frac{f_i(g)^{\lambda_i}}{\lambda_i !}=\prod_{1\leq i\leq l+k}f_i(g)^{(\lambda_i)},$$ 
where $f_i(g)=f_{\epsilon_i}(g)$ and 
$\epsilon_i=(0,\ldots , \underbrace{1}_{i-th \ place}, \ldots, 0)$ for $1\leq i\leq l+k$.
Since all above sums are finite, every element $f_i(g)$ is nilpotent. In particular, each $f_i$ maps $D$ into $(K[G]K[G]_1)_{|z_i|}$.

Consider even elements $z(g)=\sum_{1\leq i\leq l+k} z_i\otimes f_i(g)$ for $g\in G$. Then  
$$\rho^*(g)=(g\otimes 1)\sum_{k\geq 0} z(g)^{(k)}$$
for $g\in D.$
Since $\rho^*$ is a superalgebra morphism, we have $\rho^*(gh)=
\rho^*(g)\rho^*(h)$ for all $g, h\in D$. The latter holds if and only if 
$z(gh)=z(g)+z(h)$, that is  $f_i(gh)=f_i(g)+f_i(h)$ for every $i$.    

Finally, $\rho^*$ defines a supercomodule structure on $K[X]$ if and only if $\rho^*$ satisfies 
$$
\sum_{k\geq 0}((id_{K[X_u]}\otimes\Delta_G)z(g))^{(k)}=
(\sum_{k\geq 0} (z(g)\otimes 1)^{(k)})(\sum_{k\geq 0}((\rho^*\otimes id_{K[G]})z(g))^{(k)})=$$
$$
\sum_{k\geq 0}(z(g)\otimes 1 + (\rho^*\otimes id_{K[G]})z(g))^{(k)}.$$
Therefore, $K[X]$ is a supercomodule with respect to $\rho^*$ if and only if
$$(id_{K[X_u]}\otimes\Delta_G)z(g)=z(g)\otimes 1 + (\rho^*\otimes id_{K[G]})z(g).$$
In other words, for every $i$ we have the condition
$$\Delta_G(f_i(g))=f_i(g)\otimes 1 +\sum_{1\leq j\leq l+k} f_{ij}\otimes f_j(g).$$
\end{proof}

\section{First counterexample}

Suppose that $char K=0$. Let $G$ be a (connected) algebraic supergroup.    
We start with the following crucial observation.
Assume that there are nilpotent elements $f_i\in K[G]$ for $1\leq i\leq l+k$ such that $|f_i|=0$ provided $1\leq i\leq l$ and $|f_i|=1$ otherwise,
and $\Delta_G(f_i)=f_i\otimes 1+\sum_{1\leq j\leq l+k} f_{ij}\otimes f_j$ for some coefficients $f_{ij}\in K[G]$. 
Then the coefficients $f_{ij}$ satisfy
$$\Delta_G(f_{ij})=\sum_{1\leq t\leq l+k}f_{it}\otimes f_{tj} \quad \text{and} \quad \epsilon_G(f_{ij})=\delta_{ij}.$$ 
This means they define a $G$-supermodule with a basis $z_i$ for $1\leq i\leq l+k$ such that $|z_i|=|f_i|$ via
the supercomodule map $z_i\to \sum_{1\leq j\leq l+k} z_j\otimes f_{ji}.$
By Theorem \ref{actiononabelian}, we derive that the action of $G$ on the abelian supergroup $X=G_m\times G_a^l\times (G_a^-)^k$
is given by the rule
$$\rho^*(z_i)=\sum z_j\otimes f_{ji}, \rho^*(t^n)=\sum_{\lambda\geq 0} n^{|\lambda|} t^n z^{\lambda}\otimes f^{(\lambda)},$$
where $f^{(\lambda)}=\prod_{1\leq i\leq l+k}f_i^{(\lambda_i)}$, $t$ is a group-like generator of $K[G_m]$
and $n\in\mathsf{Z}$.

Set $G\simeq (G_a^-)^2$ and denote by $u$ and $v$ primitive generators of $K[G]$. If $f_1=uv, f_2=u, f_3=v,$
then
$$\Delta_G(f_i)= f_i\otimes 1+\sum_{1\leq j\leq 3} f_{ij}\otimes f_j,$$ 
where
$$\left(\begin{array}{ccc}
f_{11} & f_{12} & f_{13} \\
f_{21} & f_{22} & f_{23} \\
f_{31} & f_{32} & f_{33} \\
\end{array}\right) =
\left(\begin{array}{ccc}
1 & -v & u \\
0 & 1 & 0 \\
0 & 0 & 1 \\
\end{array}\right).
$$
This implies that $G$ acts on $X=G_m\times G_a\times (G_a^-)^2$. Define a semi-direct product
$H=X\rtimes G$.
\begin{lemma}\label{none}
No non-trivial supersubgroup of $X_u$ is normal in $H$.
\end{lemma}
\begin{proof}
Consider a non-trivial supersubgroup $N\unlhd X_u$. 
Then $K[X/N]$ is a Hopf supersubalgebra of $K[X]$. Since $X_u/N$ is abelian and unipotent, $K[X/N]$ is generated by $t$ and by a supersubspace $V$ of $W=\oplus_{1\leq i\leq 3} Kz_i$. Here $V\neq W$ but it is possible that $V=0$. It remains to prove that the superideal $I_N=K[X]V+K[X](t-1)$ does not satisfy
the condition $\rho^*(I_N)\subseteq I_N\otimes K[G]$. 
The element $\rho^*(t-1)$ equals $z_1\otimes uv +z_2\otimes u +z_3\otimes v$ modulo the superideal $I_N\otimes K[G]$.
Since $uv, u, v$ are linearly independent in $K[G]$, the condition $\rho^*(t-1)\in I_N\otimes K{G}$ implies that each $z_i$ belongs to $V$. Therefore $V=W$, which is a contradiction.        
\end{proof}
\begin{corollary}\label{unipradofH}
$H_u=1$. 
\end{corollary}
\begin{proof}
Since $H_u\bigcap X=1$, $XH_u\simeq X\times H_u$ is a normal abelian supersubgroup of $H$ and $H_u$ is isomorphic to a supersubgroup of $G$. Thus $H_u=(G^-_a)^s$, where $s=1$ or $s=2$. If $s=2$, then $H=XH_u$, which is obviously impossible. 
If $s=1$, then some proper supersubgroup $R$ of $G$ acts identically on $X$. 
If $$\rho^*(t)-t\otimes 1=tz_1\otimes uv +tz_2\otimes u +tz_3\otimes v\in K[X]\otimes I_R,$$ then 
$u, v$ and $uv$ belong to $I_R$. Therefore $R=G$, which is a contradiction, and the proof is complete.    
\end{proof}
Finally, we note that $H_{ev}=X_{ev}=G_m\times G_a$ and $(H_{ev})_u=G_a$.

\section{Algebraic Lie supersubalgebras}

Throughout this section we assume that $char K=0$ unless otherwise stated.
Let $H$ be a connected algebraic supergroup. Let $\mathfrak{S}$ be a supersubalgebra of 
$\mathsf{Lie}(H)=\mathfrak{H}$. We say that $\mathfrak{S}$ is {\it algebraic supersubalgebra} of $H$ if there is an supersubgroup $S$ of $H$ such that $\mathsf{Lie}(S)=\mathfrak{S}$. In a broader sense, a Lie superalgebra $\mathfrak{R}$ is called {\it algebraic} if  there is an algebraic supergroup $R$ such that $\mathsf{Lie}(R)=\mathfrak{R}$.
\begin{prop}\label{description}
A supersubalgebra $\mathfrak{S}$ of $\mathfrak{H}$ is algebraic if and only if $\mathfrak{S}_0$ is an algebraic subalgebra of $\mathfrak{H}_0$. 
\end{prop}
\begin{proof}
The necessary condition is obvious. 
By theorem 3.6 of \cite{carfior}, there is an equivalence between the category of connected algebraic supergroups and 
the category of connected {\it Harish-Chandra superpairs} given by $H\mapsto (H_{ev}, \mathfrak{H})$, where $H_{ev}$ acts on $\mathfrak{H}$ via the adjoint action
(see also Proposition 4.13 of \cite{mas}). Let $S_{ev}$ be  a connected subgroup of $H_{ev}$ such that
$\mathsf{Lie}(S_{ev})=\mathfrak{S}_0$. Since $S_{ev}$ is connected, we derive that $\mathfrak{S}_1$ is a $S_{ev}$-submodule of $\mathfrak{H}_1|_{S_{ev}}$ 
(cf. Lemma 7.15 of Part I in \cite{jan}). 
Therefore there is an embedding of Harish-Chandra superpairs $(S_{ev}, \mathfrak{S})\to (H_{ev}, \mathfrak{H})$. 
If $S$ is a connected algebraic supergroup whose Harish-Chandra superpair coincides with $(S_{ev}, \mathfrak{S})$, then there is
a morphism $\phi : S\to H$ such that the induced Lie superalgebra morphism $\mathsf{Lie}(S)\to \mathfrak{H}$ is an isomorphism onto $\mathfrak{S}$.
The supersubgroup $Im \phi$ of $H$ is such that its Lie superalgebra is $\mathfrak{S}$.   
\end{proof}
\begin{corollary}\label{evenpartcoincideswithitscommutant}
If $\mathfrak{S}'_0=\mathfrak{S}_0$, then $\mathfrak{S}$ is algebraic in $\mathfrak{H}$. In particular, 
if $\mathfrak{S}_0$ is semisimple, then $\mathfrak{S}$ is algebraic in $\mathfrak{H}$.
\end{corollary}
\begin{proof}
The proof follows from Theorem 6.1 of \cite{hoch} (see also Proposition 2.6 of part II, \S 6 in \cite{dg}).
\end{proof}
\begin{corollary}\label{coincidenceofcommutants}
Let $G$ be connected and $\mathfrak{G}=\mathsf{Lie}(G)$. Then $\mathsf{Lie}(G')=\mathfrak{G}'$ if and only if
$\mathfrak{G}'_0 +\mathfrak{G}'_1$ is an algebraic subalgebra of $\mathfrak{G}_0$. 
\end{corollary}
\begin{proof}
The necessary condition is obvious.
Conversely, Proposition \ref{description} implies that $\mathfrak{G}'=\mathsf{Lie}(S)$ for a connected supersubgroup $S$ of $G$.  
Using Lemma \ref{normal}, Proposition 9.1 of \cite{zub}, Lemma \ref{factorofconnect} and Lemma \ref{whenLiesuperalgisabelian}
we infer that $S\unlhd G$ and $G/S$ is an abelian supergroup.
Thus $G'\leq S$ and $\mathsf{Lie}(G')\subseteq\mathfrak{G}'$. Applying Proposition 9.1 of \cite{zub} again we obtain that $\mathfrak{G}'\subseteq \mathsf{Lie}(G')$.
Therefore $\mathsf{Lie}(G')=\mathfrak{G}'=\mathsf{Lie}(S)$. Since both $G'$ and $S$ are connected, we conclude that $G'=S$.
\end{proof}
\begin{rem}\label{criterion}
One can reformulate Corollary \ref{coincidenceofcommutants} in the following way. The Lie superalgebra $\mathsf{Lie}(G')$ is different from $\mathfrak{G}'$ if and only if there is an element $x\in \mathfrak{G}_1$ such that $[x, x]$ generates a  subalgebra in $\mathsf{Lie}(G_{ev}/G_{ev}')=\mathfrak{G}_0/\mathfrak{G}_0'$ that is not algebraic.  
\end{rem}
Let $\mathfrak{L}$ be a Lie superalgebra.
An {\it automorphism group functor} $Aut(\mathfrak{L}_a)$ is defined via
$$Aut(\mathfrak{L}_a)(A)=\{g\in GL(\mathfrak{L})(A) | \forall x, y\in\mathfrak{L}_a(A), g[x, y]=[gx, gy]\}.$$ 
\begin{prop}\label{der}
$Aut(\mathfrak{L}_a)$ is an algebraic supergroup of $GL(\mathfrak{L})$ and \linebreak \noindent
$\mathsf{Lie}(Aut(\mathfrak{L}_a))=Der (\mathfrak{L})$.
Therefore, for any Lie superalgebra $\mathfrak{L}$, its superalgebra of (super)derivations 
$Der(\mathfrak{L})$ is algebraic. 
\end{prop} 
\begin{proof}
Since any $g\in GL(\mathfrak{L})(A)$ is $A$-linear, $g\in Aut(\mathfrak{L}_a)(A)$ if and only if 
$g[x, y]=[gx, gy]$ for all $x, y\in\mathfrak{L}\otimes 1$. Choose a homogeneous basis $x_1, \ldots , x_{m+n}$ of 
$\mathfrak{L}$ such that $|x_i|=0$ for $1\leq i\leq m=\dim\mathfrak{L}_0$ and $|x_i|=1$ for $m+1\leq i\leq m+n$.
Then $[x_i, x_j]=\sum_{1\leq k\leq m+n} c_{ijk} x_k$, where $c_{ijk}=0$ whenever $|x_i|+|x_j|\neq |x_k| \pmod 2$.
Therefore $g\in Aut(\mathfrak{L}_a)(A)$ if and only if the following (homogeneous) equation
$$\sum_{1\leq k\leq m+n} c_{ijk}g_{sk}=\sum_{1\leq t, r\leq m+n}(-1)^{|x_r|(|x_i|+|x_t|)}c_{trs}g_{ti}g_{rj}$$
is satisfied for every $i, j, s$. This proves the first statement. 

Using the $\epsilon$-technique it is easy to see that  
$u\in\mathsf{Lie}(Aut(\mathfrak{L}_a))\subseteq gl(\mathfrak{L})$ if and only if $id_{\mathfrak{L}} +\epsilon_0 u_0 +\epsilon_1 u_1
\in Aut(\mathfrak{L}_a)(K[\epsilon_0, \epsilon_1])$.
The latter statement is equivalent to
$$[x, y]+u_0[x, y]\otimes\epsilon_0 +u_1[x, y]\otimes\epsilon_1 =$$
$$[u_0 x, y]\otimes\epsilon_0 +(-1)^{|y|}[u_1 x, y]\otimes\epsilon_1 +[x, u_0 y]\otimes\epsilon_0 +[x, u_1 y]\otimes\epsilon_1.$$ 
Conseqently, $u_0$ is a right even superderivation of $\mathfrak{L}$ and $u_1$ is an odd superderivation of $\mathfrak{L}$.
\end{proof}
\begin{rem}
Proposition \ref{der} can be generalized for a superalgebra with arbitrary $K$-linear operations over a field of any characteristic.
\end{rem}

\section{Second counterexample}

Let us again assume that $char K=0$.  Let $G$ be a connected algebraic group, $\mathfrak{G}=\mathsf{Lie}(G)$, and $V=Kv$ be a one-dimensional $G$-module.
Then $V$ is a $K[G]$-comodule via $v\to v\otimes g$, where $g$ is a group-like element from $K[G]$.

The differential of this action is given by $xv=\alpha(x)v$, where $\alpha\in \mathfrak{G}^*$ is defined as $\alpha(x)=
x(g)$ for $x\in\mathfrak{G}$. Clearly, $\alpha([\mathfrak{G}, \mathfrak{G}])=0$. 

To define a Harish-Chandra pair, it remains to define a Lie superalgebra structure on 
$\mathfrak{L}=\mathfrak{G}\oplus V$ such that $\mathfrak{L}_0=\mathfrak{G}$, $\mathfrak{L}_1 =V$ and
$[x, v]=\alpha(x)v$ for $x\in \mathfrak{G}$. 

It is easy to verify that a bilinear map 
$\mathfrak{L}\times \mathfrak{L}\to \mathfrak{L}$ given by $(x, y)\mapsto [x, y]$ for $x, y\in \mathfrak{L}$
defines such a structure if and only if $\alpha(a)=0$ and $[x, a]=2\alpha(x)a$, where $a=[v,v]$, for $x\in \mathfrak{G}$.
In other words, if there is an element $a\in\mathfrak{G}$ that satisfies these conditions, then there is an algebraic supergroup $L$ such that $L_{ev}\simeq G$ and $\mathsf{Lie}(L)=\mathfrak{L}$. 

If $\alpha\neq 0$, then $a\in\mathfrak{L}'_0$. By Corollary \ref{coincidenceofcommutants} we have $\mathfrak{L}'=\mathsf{Lie}(L')$.
In the case $\alpha=0$ one can choose $a$ to be any central element in $\mathfrak{G}$. In particular, if $a$ is
not algebraic modulo $\mathfrak{G}'$, then $\mathfrak{L}'\neq\mathsf{Lie}(L')$ according to Remark \ref{criterion}.   
Below we give an example of this kind.

Consider a supersubgroup $L$ of $GL(2|2)$ such that for any $C\in \mathsf{SAlg}_K$ the group $L(C)$
consists of all matrices 
$$\left(\begin{array}{cc}
A & B \\
B & A
\end{array}\right)$$
with the even and odd blocks given as 
$$A=\left(\begin{array}{cc} 
\alpha_1 & \alpha_2 \\
0 & \alpha_1
\end{array}\right), B=\left(\begin{array}{cc}
t & (1+\alpha_1^{-1}\alpha_2)t \\
0 & t\end{array}\right),$$
where $\alpha_1\in C_0^{\times}, \alpha_2\in C_0$ and $t\in C_1.$ 

Obsrve that $AB=BA$ and any two matrices $B$ and $B'$, as above, satisfy $BB'=-B'B$.
In what follows we represent each element of $L(C)$ as a row $(A|B)$. In this notation
$$(A|B)(A'|B')=(AA'+BB'|AB'+BA'), \ (A|B)^{-1}=(A^{-1}|-A^{-2}B).$$
\begin{lemma}\label{commutator}
If $(A|B)$ and $(A'|B')$ represent elements from $L(C)$, then
$[(A|B), (A'|B')]=(E+2A^{-1}A'^{-1}BB'|0)$.
\end{lemma}
\begin{proof}
Proof follows by straightforward calculations.
\end{proof}
It follows from the above Lemma that $L_{ev}$ is central in $L$. 
\begin{lemma}
If $L$ is as above, then $L'=L_{ev}$.
\end{lemma}
\begin{proof}
Since $L/L_{ev}\simeq G_a^-$, we obtain that $L'\leq L_{ev}$. On the other hand,
$L_{ev}\simeq G_m\times G_a$ and this decomposition corresponds to the Jordan-Chevalley
decomposition 
$$\left(\begin{array}{cc} 
\alpha_1 & \alpha_2 \\
0 & \alpha_1
\end{array}\right)=
\left(\begin{array}{cc} 
\alpha_1 & 0 \\
0 & \alpha_1
\end{array}\right)
\left(\begin{array}{cc} 
1 & \alpha_1^{-1}\alpha_2 \\
0 & 1
\end{array}\right).$$
Since $L'$ is a connected, there are three possibilities, either $L'=G_m, L'=G_a$ or $L'=L_{ev}$. 
In first two cases, the elements of $L'$ should satisfy
either $\alpha_2=0$ or $\alpha_1-1=0$. Using Lemma \ref{commutator} it is easy to find a commutator element that does not satisfy neither
first equation nor second one.   
\end{proof}
The Lie superalgebra $\mathfrak{L}$ is generated by the elements $x=(E|0), y=(E_{12}|0), v=(0|E+E_{12})$ which satisfy the relations $[x, y]=[x, v]=[y, v]=0$ and $a=[v, v]=2x+4y$. 
Therefore $\mathfrak{L}'$ is a proper supersubalgebra of $\mathsf{Lie}(L')$. Additionally, the element $a=2x+4y$ does not generate an algebraic subalgebra of $\mathfrak{L}_0$. 

\section{Reductive and quasireductive supergroups}    
In this section we assume that $K$ is algebraically closed field of characteristic zero.
The following proposition was proved in \cite{zubul}.
\begin{prop}\label{somealgebraic}
Let $\mathfrak{S}$ be either the solvable radical of $\mathfrak{H}$ or a maximal abelian ideal of $\mathfrak{H}$.
Then $\mathfrak{S}$ is algebraic.
\end{prop}
A connected normal supersubgroup $R=R(H)$ of $H$ for which $\mathfrak{R}=\mathsf{Lie}(R)$ is the solvable radical of $\mathfrak{H}$ is called a {\it solvable radical} of $H$. The supergroup $R(H)$ is actually solvable (cf. \cite{zubul}).
\begin{lemma}\label{unipradofsolv}
If $H$ is a reductive supergroup that does not contain non-trivial central toruses, then $R(H)=1$ and $\mathsf{Lie}(H)=\mathfrak{H}$ is semisimple.  
\end{lemma} 
\begin{proof}
Since the supersubgroup $R(H)^{(k)}$, for sufficiently large positive integer $k$, is non-trivial, normal, connected and abelian, we can apply Theorem \ref{actiononabelian}.   
\end{proof}
\begin{lemma}\label{directsum}
Let $G$ be a connected algebraic supergroup. Suppose that its Lie superalgebra $\mathfrak{G}$ is a direct sum of centerless superideals 
$\mathfrak{G}_i$ for $1\leq i\leq t$. Then each $\mathfrak{G}_i$ is algebraic in $\mathfrak{G}$. 
\end{lemma}
\begin{proof}
Fix an index $i$. Lemma 9.4 of \cite{zub} implies that $\oplus_{j\neq i}\mathfrak{G}_j$ is an $G$-supersubmodule
of $\mathfrak{G}$ with respect to the adjoint action of $G$. Denote the kernel of the induced supergroup morphism
$G\to GL(\oplus_{j\neq i}\mathfrak{G}_j)$ by $G_i$.
By Lemma 3.4 and Proposition 9.1 of \cite{zub}, the superalgebra
$\mathsf{Lie}(G_i)$ equals $\oplus_{j\neq i} Z(\mathfrak{G}_j)\oplus \mathfrak{G}_i=\mathfrak{G}_i$. 
\end{proof}
\begin{rem}\label{groupproduct}
According to Lemma \ref{normal}, each $G_i$ is normal in $G$ and for any pair of indexes $i\neq j$ the supersubgroup
$G_i\bigcap G_j$ is finite. Moreover, the induced morphism $\pi : G_1\times\ldots\times G_t\to G$ is a {\it quasi-isomorphism},
that is, $\pi$ is surjective and $\ker\pi$ is finite. 
\end{rem}
We will review the classification of semi-simple Lie superalgebras from \cite{kac}. 
Fix a collection of simple Lie superalgebras $\mathfrak{U}_1, \ldots , \mathfrak{U}_t$ and consider the Lie superalgebra
$\mathfrak{U}=\oplus_{1\leq i\leq t} \mathfrak{U}_i\otimes Sym(n_i)$, where
$Sym(n)$ is the symmetric superalgebra defined on a superspace of superdimension $0|n$.  

It is known that
$$Der(\mathfrak{U})=\oplus_{1\leq i\leq t}((
Der(\mathfrak{U}_i)\otimes Sym(n_i))\oplus (id_{\mathfrak{U}_i}\otimes Der(Sym(n_i))),$$ 
where
$$(x\otimes a)(\delta_1\otimes b)=(-1)^{|a||\delta_1|}x\delta_1\otimes ab, \quad (x\otimes a)(id_{\mathfrak{U}_i}\otimes\delta_2)=x\otimes a\delta_2,$$
$\delta_1\in Der(\mathfrak{U}_i), \delta_2\in Der(Sym(n_i)), x\in \mathfrak{U}_i, a, b\in Sym(n_i)$ and $1\leq i\leq t$.
If $z_1, \ldots, z_n$ are free (odd) generators of $Sym(n)$, then $Der(Sym(n))$ 
consists of all (right) superderivations 
$\sum_{1\leq i\leq n}\frac{d}{dz_i}f_i$ for $f_i\in Sym(n)$.
Moreover, $Der(Sym(n))$ has a consistent $\mathsf{Z}$-grading where
$$Der(Sym(n))(k)=\{\sum_{1\leq i\leq n}\frac{d}{dz_i}f_i|
f_i\in Sym(n), \deg f_i=k+1, 1\leq i\leq n\}$$
for each $k\in\mathsf{Z}$.
 
The superalgebra $\mathfrak{U}$ is identified with the superalgebra of
inner derivations $Inder(\mathfrak{U})\subseteq Der(\mathfrak{U})$. If a Lie
supersubalgebra $\mathfrak{H}\subseteq Der(\mathfrak{U})$ contains $\mathfrak{U}$, then $\mathfrak{H}$ is
semisimple if and only if the projection of $\mathfrak{H}$ to $id_{\mathfrak{U}_i}\otimes
Der(Sym(n_i))(-1)$ is onto for each $i$. Moreover, any
semisimple superalgebra $\mathfrak{H}$ is obtained this way.

Let $\mathfrak{V}$ be a simple Lie superalgebra. The following standard formulas are valid in
every $Der(\mathfrak{V}\otimes Sym(n))$.
\begin{enumerate}
\item $[\delta_1\otimes a, \delta'_1\otimes a']=(-1)^{|a||\delta'_1|}[\delta_1, \delta'_1]\otimes aa'$ for 
$\delta_1, \delta'_1\in Der(\mathfrak{V})$ and $a, a'\in Sym(n)$;
\item $[\delta_1\otimes a, id_{\mathfrak{V}}\otimes\delta_2 ]=\delta_1\otimes a\delta_2$ for $\delta_1\in Der(\mathfrak{V})$,
$\delta_2\in Der(Sym(n))$ and $a\in Sym(n);$
\item $[id_{\mathfrak{V}}\otimes\delta_2, id_{\mathfrak{V}}\otimes\delta'_2]=id_{\mathfrak{V}}\otimes [\delta_2, \delta'_2]$ for
$\delta_2, \delta'_2\in Der(Sym(n))$.
\end{enumerate} 
\begin{lemma}\label{firstreduction}
Let $\mathfrak{V}$ be a simple Lie superalgebra. Then the superalgebra $Der(\mathfrak{V})\otimes Sym(n)$ is algebraic in $Der(\mathfrak{V}\otimes Sym(n))$.
\end{lemma}
\begin{proof}
Let $H$ be an algebraic supergroup such that $\mathsf{Lie}(H)=Der(\mathfrak{V}\otimes Sym(n))$. Since $Der(\mathfrak{V})\otimes Sym(n)$ is a superideal of $Der(\mathfrak{V}\otimes Sym(n))$, one can argue as in Lemma \ref{directsum} to define the induced morphism 
$$H\to GL(Der(\mathfrak{V}\otimes Sym(n))/(Der(\mathfrak{V})\otimes Sym(n)))=GL(Der(Sym(n))).$$
Denote the kernel of this morphism by $R$. Then the superalgebra \newline \noindent
$\mathsf{Lie}(R)/(Der(\mathfrak{V})\otimes Sym(n))$ is isomorphic to the center of $Der(Sym(n))$, which is trivial. Thus $\mathsf{Lie}(R)=Der(\mathfrak{V})\otimes Sym(n)$.
\end{proof}
\begin{lemma}\label{firstalg}
Let $\mathfrak{V}$ be a Lie superalgebra. If $\mathfrak{V}_0 =\mathfrak{V}'_0$, then $\mathfrak{V}\otimes Sym(n)$ is algebraic in $Der(\mathfrak{V}\otimes Sym(n))$. 

In particular, the conclusion holds for any simple superalgebra $\mathfrak{V}$.
\end{lemma}
\begin{proof}
If $\mathfrak{V}=\mathfrak{V}'$, then $[\mathfrak{V}_0 , \mathfrak{V}_1]=\mathfrak{V}_1$. 
Therefore $(\mathfrak{V}\otimes Sym(n))'_0 =(\mathfrak{V}\otimes Sym(n))_0$ and the first statement follows by Corollary \ref{evenpartcoincideswithitscommutant}. 

Using Table III on p.13 of \cite{frsorsci}, we see that the second statement follows immediately for all cases except 
$A(m, n), C(n+1), W(n), S(n), \tilde{S}(n)$ and $H(n)$. 

In the cases $A(m, n), C(n+1), W(n)$ and $\tilde{S}(n)$, Proposition 5.1.2 of \cite{kac} shows that $Der(\mathfrak{V})=\mathfrak{V}$ and our claim follows by Lemma \ref{firstreduction}. 

If $\mathfrak{V}$ is of type $S(n)$ or $H(n)$, then
$\mathfrak{V}_0=\mathfrak{G}\oplus\mathfrak{R}$, where $\mathfrak{G}$ is isomorphic to $sl(n)$ or $so(n)$ and the solvable radical $\mathfrak{R}$ is a direct sum of non-trivial simple $\mathfrak{G}$-modules (see \cite{frsorsci}, p.6-7). 
Thus $\mathfrak{G}'=\mathfrak{G}$ and $[\mathfrak{G}, \mathfrak{R}]=\mathfrak{R}$, and this implies $\mathfrak{V}'_0=\mathfrak{V}_0$. 
\end{proof}
\begin{prop}\label{resume}
The superalgebra $\mathfrak{U}$, defined above, is algebraic in $Der(\mathfrak{U})$.
\end{prop}
\begin{proof}
Let $G$ be an algebraic supergroup such that $\mathsf{Lie}(G)=Der(\mathfrak{U})$. 
Using the formulas (1)-(3) it can be easily shown that both $Der(\mathfrak{U}_i\otimes Sym(n_i))$ and  $\mathfrak{U}_i\otimes Sym(n_i)$ do not have non-zero central elements. By Lemmas \ref{directsum} and \ref{firstalg}, there are normal supersubgroups
$U_i$ and $G_i$ of $G$ such that $\mathsf{Lie}(G_i)=Der(\mathfrak{U}_i\otimes Sym(n_i))$ and $\mathsf{Lie}(U_i)=\mathfrak{U}_i\otimes Sym(n_i)$ for $1\leq i\leq t$. 
It remains to set $U=U_1\ldots U_t$ and use Lemma \ref{productofnormalandsome}. 
\end{proof}
Let us call the pair of supergroups $(U, G)$, where $\mathsf{Lie}(U)=\mathfrak{U}$, $\mathsf{Lie}(G)=Der(\mathfrak{U})$ and $\mathfrak{U}$ as previously defined,
a {\it sandwich pair}. We say that a connected algebraic supergroup $H$ is inserted into a sandwich pair $(U, G)$ if there is a supergroup morphism $\phi : H\to G$ such that $U\leq\phi(H)$ and $\ker\phi$ is finite. 
\begin{theorem}\label{sandwich}
Let $H$ be a supergroup such that $\mathfrak{H}=\mathsf{Lie}(H)$ is semisimple. Then $H$ is inserted into a sandwich pair. 
\end{theorem}
\begin{proof}
Since $\mathfrak{H}$ is semisimple, there is a Lie superalgebra $\mathfrak{U}$ as previously defined such that
$\mathfrak{U}\subseteq\mathfrak{H}\subseteq Der(\mathfrak{U})$.
By Lemma \ref{AdcommuteswithLieoperation}, the map $\mathsf{Ad} : H\to GL(\mathfrak{H})$ preserves the Lie operation in $\mathfrak{H}_a$. In particular, the 
induced morphism $\phi : H\to GL(\mathfrak{U})$ maps $H$ to $Aut(\mathfrak{U}_a)$. By Proposition 9.1 of \cite{zub},
$\mathsf{Lie}(\ker\phi)$ coincides with the centralizer of  $\mathfrak{U}$ in $\mathfrak{H}$. Using Proposition 5.1.2 of \cite{kac} 
and formulas (1)-(3), one can easily show that $Z_{Der(\mathfrak{U})}(\mathfrak{U})$ is zero.  
\end{proof}

The following question arises naturally: Is every semisimple superalgebra $\mathfrak{H}$, containing $\mathfrak{U}$ and contained in $Der(\mathfrak{U})$, algebraic? 
The answer is negative in general. Moreover, the following (very elementary!) example is an evidence that many of such superalgeras $\mathfrak{H}$ are not algebraic. 
We think it is likely that the algebraic supersubalgebras $\mathfrak{H}$ as above form an observable family. We plan to describe such a family in a future article.
\begin{exam}\label{notalgebraic}
Set $\mathfrak{U}=\mathfrak{sl}_2(K)\otimes Sym(2)$. Define the superalgebra $\mathfrak{H}$, satisfying $\mathfrak{U}\subseteq \mathfrak{H}\subseteq Der(\mathfrak{U})$, by
$$\mathfrak{H} = \mathfrak{U}\oplus K(id_{\mathfrak{U}}\otimes \frac{d}{dz_1})\oplus K(id_{\mathfrak{U}}\otimes \frac{d}{dz_2})\oplus K(id_{\mathfrak{U}}\otimes \delta),$$
where $\delta=\frac{d}{dz_1}(z_1+z_2)+\frac{d}{dz_2}z_2$. 

Suppose that $\mathfrak{H}$ is algebraic and $H$ is an algebraic supergroup such that $\mathsf{Lie}(H)=\mathfrak{H}$. 
Since $V=\mathfrak{sl}_2(K)\otimes (Kz_1+Kz_2)$ is a $\mathfrak{H}_0$-submodule of $\mathfrak{H}_1$, it is also an $H_{ev}$-submodule with respect to the adjoint action.  
This implies that the semisimple and nilpotent components of the operator $id_{\mathfrak{U}}\otimes \delta |_V$ belong to the image of $\mathfrak{H}_0$ in $\mathfrak{gl}(V)$. 
Since this is not true, we have reached a contradiction.
\end{exam}
Recall that a connected algebraic supergroup $H$ is called {\it quasireductive} if $H_{ev}$ is linearly reductive, or equivalently, $(H_{ev})_u=1$ (see \cite{wat1}).
If $H$ is quasireductive, then its Lie superalgebra $\mathfrak{H}=\mathsf{Lie}(H)$ is also quasireductive, meaning that 
$\mathfrak{H}_0$ is reductive and $\mathfrak{H}_0$-module $\mathfrak{H}$ is semisimple (cf. \cite{serg}). 
Using Proposition 9.3 of \cite{maszub} and the standard properties of reductive algebraic groups one can show that a connected normal supersubgroup is quasireductive, and a homomorphic image of quasireductive supergroup is quasireductive as well.
\begin{lemma}\label{quasiredandsemisimple}
Assume that a superalgebra $\mathfrak{H}$ is semisimple and quasireductive. Then 
$\mathfrak{U}\subseteq\mathfrak{H}\subseteq Der(\mathfrak{U})$, where $\mathfrak{U}=\oplus_{1\leq i\leq t} \mathfrak{U}_i\otimes Sym(n_i)$ 
is such that every summand $\mathfrak{U}_i\otimes Sym(n_i)$ is one of the following two types:

- $n_i=0$ and $\mathfrak{U}_i$ is a classical simple Lie superalgebra or

- $n_i=1$ and $\mathfrak{U}_i$ is a simple Lie algebra.

\end{lemma}
\begin{proof}
Since every superideal of a quasireductive Lie superalgebra is also quasireductive, it is enough to consider the case $\mathfrak{U}=\mathfrak{V}\otimes Sym(n)$, where $\mathfrak{V}$ is a simple Lie superalgebra. The ideal $\mathfrak{V}_0\otimes Sym(n)^2_1 +\mathfrak{V}_1\otimes Sym(n)_1$
is nilpotent, and therefore it is contained in the center of the reductive algebra $\mathfrak{U}_0$. If $Sym(n)_1^2\neq 0$, then
$\mathfrak{V}_0\subseteq Z(\mathfrak{V})$ which implies $\mathfrak{V}_0=0$ and $\mathfrak{V}$ is abelian.
This contradiction implies $Sym(n)_1^2 =0$, and therefore $n=0, 1$. If $n=1$, then $[\mathfrak{V}_0, \mathfrak{V}_1]=0$. Assuming $\mathfrak{V}_1\neq 0$, we get
$\mathfrak{V}_0=\mathfrak{V}'_1$ and $\mathfrak{V}_0$ is abelian. Consequently, $\mathfrak{V}_0\subseteq Z(\mathfrak{V})$ and $\mathfrak{V}_1 =0$. 
Finally, if $n=0$, then $\mathfrak{V}$ is quasireductive and therefore $\mathfrak{V}$ is classical (cf. \cite{frsorsci}).   
\end{proof}
\begin{lemma}
A connected algebraic supergroup $H$ is quasireductive if and only if $R(H)_{ev}$ is a torus and $H/R(H)$ is quasireductive.
\end{lemma}
\begin{proof}
If $H$ is quasireductive, then $R(H)_{ev}$ is a connected normal subgroup of the reductive group $H_{ev}$ and therefore it is also reductive. Since $R(H)_{ev}$ is solvable, it is a torus. Conversely, $(H_{ev})_u$ is isomorphic to a subgroup of
$(H_{ev}/R(H)_{ev})_u =1$. 
\end{proof}
From now on, assume that $H$ is a (connected) quasireductive supergroup and $\mathfrak{H}=\mathsf{Lie}(H)$.
\begin{lemma}\label{centerisalgebraic}
Let $\mathfrak{Z}=Z(\mathfrak{H})$. Then $\mathfrak{Z}$ is algebraic in $\mathfrak{H}$.
\end{lemma}
\begin{proof}
Let $\tilde{\mathfrak{Z}}$ be the preimage of $Z(\mathfrak{H}/\mathfrak{Z})$. We have already seen that 
$\tilde{\mathfrak{Z}}$ is algebraic in $\mathfrak{H}$. By Lemma 5.5 of \cite{serg} we infer $\mathfrak{Z}_0 =\tilde{\mathfrak{Z}}_0$.
The claim follows using Proposition \ref{description}.  
\end{proof}
Denote by $Z$ the connected supersubgroup of $H$ such that $\mathsf{Lie}(Z)=\mathfrak{Z}$.
\begin{lemma}\label{center}
The supergroup $Z$ satisfies $Z=Z(H)^0$. Moreover, the largest central torus of $H$ coincides with the largest central torus of $R(H)$. 
\end{lemma}
\begin{proof}
From Lemma \ref{centralsubgroup} we infer that $Z\subseteq Z(H)^0$. Since $\mathsf{Lie}(Z(H)^0)\subseteq\mathfrak{Z}$, the first statement follows.

Denote the largest central toruses of $H$ and $R(H)$ by $T$ and $T'$, respectively. It is clear that $T\leq T'$. On the other hand, $Z(R(H))^0$ is a normal supersubgroup of $H$ (see \cite{zubul}). Using Theorem \ref{actiononabelian} we conclude that $T'=Z(R(H))^0_{ev}\leq T$.
\end{proof}
\begin{lemma}\label{solvablequasired}
Let $T$ be the largest central torus of a solvable supergroup $H$. Then $H/(T\times H_u)$ is a triangulizable supergroup. 
\end{lemma}
\begin{proof}
We proceed by induction on $\dim\mathfrak{H}$. For sufficiently large positive integer $k$, the supersubgroup $H^{(k)}$ decomposes as $T'\times
U$, where $T'$ is torus, $U$ is odd unipotent and $T'\neq 1$ or $U\neq 1$. By Theorem \ref{actiononabelian}, $T'$
is central and contained in the largest central torus $T$ of $H$. Moreover, $T\times H_u$ is a non-trivial normal connected supersubgroup. 
Thus $H/(T\times H_u)$ is either triangulizable or it contains a non-trivial central torus $T''$.     
Denote by $R$ the preimage of $T''$ in $H$. Then $R\unlhd H$ and $R_{ev}$ is a torus. Since $R_{ev}/T\simeq T''$, we have
$R=R_{ev}\times H_u$ that implies that $R_{ev}$ is a central torus bigger than $T$. This contradiction completes the proof.   
\end{proof}
\begin{rem}
The statement of Lemma \ref{solvablequasired} can be reformulated as follows: If $H$ is solvable, then 
$H/Z$ is triangulizable. Additionally, observe that $H_u$ is always an odd supergroup.
\end{rem}
The supergroup $\tilde{H}=H/R(H)$ can be inserted in a sandwich pair $(U, G)$. By Proposition \ref{resume} and Lemma \ref{quasiredandsemisimple}, we have 
$U=U_1\ldots U_t$ and $G=G_1\ldots G_t$, where $\mathsf{Lie}(U_i)=\mathfrak{U}_i\otimes Sym(n_i),
\mathsf{Lie}(G_i)=Der(\mathfrak{U}_i\otimes Sym(n_i))$ and $n_i=0, 1$ are such that $\mathfrak{U}_i$ is a classical simple superalgebra provided $n_i=0$
and $\mathfrak{U}_i$ is a simple Lie algebra if $n_i=1$.

Order the components $U_i$ so that all superalgebras $\mathfrak{U}_i$ of type $A(n, n), P(n)$ or $Q(n)$ correspond to indices $1\leq i\leq s$, and
components with $n_i=0$ correspond to indices $s+1\leq i\leq s'$, where $s'\leq t$. According to Proposition 5.1.2(a) of \cite{kac} we have 
$U_i=G_i$ for $i > s'$. Then 
$$\mathsf{Lie}(G/U)=(\oplus_{1\leq i\leq s} Der(\mathfrak{U}_i)/\mathfrak{U}_i)\oplus
Der(Sym(1))^{\oplus (s'-s)}.$$

Denote the superalgebra $\mathsf{Lie}(G/U)$ by $\mathfrak{F}$ and supergroups $G_i U/U\simeq G_i/U_i$ by $R_i$ for $1\leq i\leq s'$.

\begin{lemma}\label{structureofGmoduloU}
Using the above notations, we have the following statements.
\begin{enumerate}
\item If $\mathfrak{U}_i$ is of type $A(n,n)$ or $P(n)$, then $R_i\simeq G_m$. If $\mathfrak{U}_i$ is of type $Q(n)$, 
then $R_i\simeq G_a^-$.
\item $Z(G/U)^0=R_1\ldots R_s$.
\item If $i>s$, then $R_i'\simeq G_a^-$ and $R_i/R_i'\simeq G_m$.
\end{enumerate} 
\end{lemma}
\begin{proof}
If $\mathfrak{U}_i$ is of type $A(n, n)$ or $P(n)$, then $Der(\mathfrak{U}_i)/\mathfrak{U}_i$ is one-dimensional and generated by an even derivation that acts on $\mathfrak{U}_i$ as a semisimple operator (see Proposition 5.1.2(b) of \cite{kac}). If 
$\mathfrak{U}_i$ is of type $Q(n)$, then $Der(\mathfrak{U}_i)/\mathfrak{U}_i$ is one-dimensional and odd (see Proposition 5.1.2(c) of \cite{kac}). The first statement follows. 

It is clear that $Z(\mathfrak{F})=\oplus_{1\leq i\leq s} Der(\mathfrak{U}_i)/\mathfrak{U}_i$.
Lemma \ref{factorofconnect} shows that $R_1\ldots R_s$ is connected. The second statement then follows by applications of 
Lemmas \ref{center} and \ref{centralsubgroup}.
 
Finally, $Der(Sym(1))'= Der(Sym(1))_1$ and $Der(Sym(1))_1$ is a semisimple
$Der(Sym(1))_0$-module. Corollary \ref{coincidenceofcommutants} implies the third statement.
\end{proof}
\begin{theorem}\label{finaldescription}
A connected algebraic supergroup $H$ is quasireductive if and only if all of the following conditions hold :
\begin{enumerate}
\item $R(H)_{ev}$ is a torus, or equivalently, $R(H)$ satisfies the conclusion of Lemma \ref{solvablequasired};
\item $\tilde{H}=H/R(H)$ contains normal supersubgroups $U, U_1, \ldots , U_t$ such that $U=U_1\ldots U_t$ and 
the induced morphism $U_1\times\ldots\times U_t\to U$ is a quasi-isomorphism. 
Besides, for every $1\leq i \leq t$, $\mathsf{Lie}(U_i)=\mathfrak{U}_i\otimes Sym(n_i)$ are such that either $n_i=0$ and $\mathfrak{U}_i$ is a classical simple Lie superalgebra, or $n_i=1$ and 
$\mathfrak{U}_i$ is a simple Lie algebra;
\item $\tilde{H}/U$ is a triangulizable supergroup with odd unipotent radical.  
\end{enumerate}
\end{theorem}
\begin{proof}
The necessary condition is now obvious. Conversely, if $H$ satisfies the conditions of this theorem, then 
$H_{ev}$ has normal subgroups, say $L_1$ and $L_2$ such that $H/L_1$ and $H/L_2$ are toruses and $\mathsf{Lie}(H/L_1)$ is reductive.
Therefore $H/L_1$ is linearly reductive which implies that $(H_{ev})_u=1$.  
\end{proof}

\section*{Acknowledgments}
The second author was supported by grants from the Russian Ministry of Education and Science (grants 14.V37.21.0359 and 0859) and by 
Brazilian FAPESP 12/50027-0. He thanks both of them for their support.

\end{document}